\documentclass{birkjour}
\usepackage{amssymb,amsmath,amsthm,newlfont,enumerate}
\usepackage{xcolor}

\usepackage{hyperref}
\hypersetup{
    colorlinks=true,
    linkcolor=black,     
    urlcolor=grey,
    citecolor=red
}
\theoremstyle{plain}
\newtheorem{theorem}{Theorem}[section]
\newtheorem{lemma}[theorem]{Lemma}
\newtheorem{corollary}[theorem]{Corollary}

\theoremstyle{definition}
\newtheorem{definition}[theorem]{Definition}

\newtheorem{problem}{Problem}[section]

\theoremstyle{remark}

\numberwithin{equation}{section}

\newcommand{\bC}{\mathbb{C}}
\newcommand{\bR}{\mathbb{R}}

\newcommand{\bT}{\mathbb{T}}
\newcommand{\bD}{\mathbb{D}}

\newcommand{\cH}{\mathcal{H}}

\newcommand{\oa}{\left\lbrace}
\newcommand{\fa}{\right\rbrace}

\newcommand{\op}{\left(}
\newcommand{\fp}{\right)}

\newcommand{\ra}{\rightarrow}

\DeclareMathOperator{\hol}{Hol}

\begin{document}

\date{\today}
\title[Power-series summability methods in de Branges--Rovnyak spaces]{Power-series summability methods in \\de Branges--Rovnyak spaces}

\author{Javad Mashreghi}
\address{D\'epartement de math\'ematiques et de statistique, Universit\'e Laval,
Qu\'ebec City (Qu\'ebec),  Canada G1V 0A6.}
\email{javad.mashreghi@mat.ulaval.ca}

\author{Pierre-Olivier Paris\'e}
\address{D\'epartement de math\'ematiques et de statistique, Universit\'e Laval,
Qu\'ebec City (Qu\'ebec),  Canada G1V 0A6.}

\curraddr{Department of Mathematics, University of Hawaii at Manoa,
Honolulu, Hawaii 96822, U.S.A.}
\email{parisepo@hawaii.edu}

\author{Thomas Ransford}
\address{D\'epartement de math\'ematiques et de statistique, Universit\'e Laval,
Qu\'ebec City (Qu\'ebec),  Canada G1V 0A6.}
\email{thomas.ransford@mat.ulaval.ca}

\thanks{JM supported by an NSERC Discovery Grant. 
POP supported by an NSERC Alexander-Graham-Bell Scholarship and a scholarship from FRQNT.
TR supported by grants from NSERC and the Canada Research Chairs program.
}

\begin{abstract}
We show that there exists a de Branges--Rovnyak space $\mathcal{H}(b)$ on the unit disk containing a function $f$ with the following property:
even though $f$ can be approximated by polynomials in $\mathcal{H}(b)$, neither the Taylor partial sums of $f$ nor their Ces\`aro, Abel, Borel or logarithmic means converge to $f$ in $\mathcal{H}(b)$.

A key tool is a new abstract result showing that, if one regular summability  method includes another for scalar sequences, then it automatically does so  for certain Banach-space-valued sequences too.

\end{abstract}



\maketitle

\section{Introduction}\label{S:Intro}
We denote by $\bD$ the open unit disk and $\hol (\bD )$ the space of holomorphic functions on $\bD$. Polynomial approximation in Banach spaces of holomorphic functions on the unit disk has attracted attention recently. In \cite{Mashreghi2019}, the authors proved that, for a Banach space of holomorphic functions $X$ on the open unit disk having the approximation property and containing a dense set of polynomials, there exist linear bounded operators $T_n : X \ra X$ such that, for each $f \in X$, the functions $T_n(f)$ are polynomials and $T_n(f) \ra f$ in the norm of $X$. The authors called the aforementioned sequence $(T_n)$ a \textit{linear polynomial approximation scheme}.

For some spaces, the operators $T_n$ are explicitly known. For example, if $X = H^2$, the Hardy space, then the bounded linear operators can be chosen to be the $n$-th partial sums $s_n[f]$ of the Taylor expansion of a function $f (z) = \sum_{n \geq 0} a_n z^n$ belonging to $H^2$. This is an easy consequence of the definition of the norm in $H^2$. When $X = H^p$, the Hardy spaces with $1 < p < \infty$, we can still choose the partial sums $s_n[f]$ of a function $f \in H^p$ as a linear polynomial approximation scheme. The proof is more elaborate and is based on a result of Riesz on the boundedness of the Hilbert transform. The reader is referred to \cite[p.~108]{Jevic2016}. When $X = A (\bD )$, the disk algebra, due to a slight variant of du Bois-Reymond's theorem that establishes the existence of a continuous function on the unit circle whose Fourier series diverges at one point, the partial sums $s_n[f]$ do not converge to $f$ in the norm of $A (\bD )$. Instead, we can use the Ces\`aro means of order 1 of the Taylor expansion of a function $f \in A (\bD )$, defined as
	$$
	\sigma_n [f] := \frac{1}{n + 1} \sum_{k = 0}^n s_k[f] .
	$$
This is essentially Fej\'er's theorem. This procedure also works for certain other spaces, for example the Hardy space $H^1$ and the weighted Dirichlet spaces $\mathcal{D}_\omega$ for superharmonic weights $\omega : \bD \ra (0, \infty )$.

However there are some spaces for which we do not explicitly know the linear polynomial approximation scheme. One such family of spaces are the de Branges--Rovnyak spaces $\cH (b)$, where $b \in H^\infty$ is a non-extreme point of the unit ball of $H^\infty$. Despite the fact that the set of polynomials is dense in $\cH (b)$, the authors of \cite{El-Fallah2016} showed that, for certain choices of $b$, the partial sums $s_n[f]$ and the Ces\`aro means $\sigma_n[f]$ may fail to converge in $\cH (b)$ to the initial function $f$. Therefore the attention turned to other linear summability methods that do not give a polynomial approximation, but have a better chance to approximate the function $f$ in the norm of $\cH (b)$, namely the dilates of $f$, which are in fact the Abel means of the partial sums $s_n[f]$ : 
	$$
	f_r (z) := \sum_{n \geq 0} a_n r^n z^n = (1 - r) \sum_{n \geq 0} s_n[f] (z) r^n \quad (r \in [0, 1), \, z \in \bD )
	$$
for $f(z) = \sum_{n \geq 0} a_n z^n$. Nevertheless, they showed that even this summability method, which includes all the Ces\`aro summability methods of order $\alpha > -1$, may fail for de Branges--Rovnyak spaces with non-extreme symbols $b$. Furthermore, the same authors showed in the same article that there is a constructive way to obtain a polynomial approximation of a given function $f \in \mathcal{H} (b)$ for any non-extreme point $b$ of the unit ball of $H^{\infty}$. However, the procedure is highly non-linear and it does not correspond to a polynomial approximation scheme.

In this article, we study another summability method which includes the Abel method : the logarithmic method. Its means are defined as
	\begin{equation}
	L_r[f] (z) := \frac{r}{\log \op \frac{1}{1 - r} \fp} \sum_{n \geq 0} \frac{1}{n + 1} s_n[f] (z) r^n
	\end{equation}
for $f(z) = \sum_{n \geq 0} a_n z^n$ and for $r \in [0, 1)$, $z \in \bD$. It was introduced by Borwein in \cite{Borwein1957a} through a power-series method. 

If $f \in \cH (b)$, it is well known that $f_r \in \cH (b)$. We will show that also $L_r [f] \in \cH (b)$. However, our main result reveals that $L_r[f]$ may diverge.
	\begin{theorem}\label{T:MainresultLogDivergence}
	There exist a non-extreme point $b$ of the unit ball of $H^{\infty}$ and a function $f \in \mathcal{H} (b)$ such that
		$$
		\lim_{r \ra 1^-} \Vert L_r[f] \Vert_b = \infty .
		$$
	\end{theorem}
Theorem \ref{T:MainresultLogDivergence} is proved in \S\ref{SS:DivergenceOfLogInHb}. 
It is a consequence of the example already constructed in \cite{El-Fallah2016} 
to show that $s_n[f]$ and $\sigma_n[f]$ may diverge in $\cH(b)$, together with an
integral formula that links the dilates of a function to its logarithmic means.

As a corollary, we obtain the following result concerning generalized Abel methods $A_r^\alpha$, where $\alpha > -1$ and where the generalized Abel means $A_r^\alpha [f]$ are defined by the following expression: 
	$$
	A_r^\alpha [f] (z) := (1 - r)^{1 +\alpha} \sum_{n \geq 0} \binom{n+ \alpha}{\alpha} s_n[f](z) r^n  \quad (r \in [0, 1) ).
	$$
	\begin{corollary}\label{C:Mainresult2AbelDivergence}
	For every $\alpha > -1$, there exist a non-extreme point $b$ of the unit ball of $H^\infty$ and a function $f \in \cH (b)$ such that $A_r^\alpha [f] \not\ra f$ in $\cH (b)$ as $r \ra 1^-$.
	\end{corollary}	
This corollary is a consequence of an abstract result in functional analysis. It enables us to compare the summability of a sequence of elements in a Banach space with respect to two summability methods, based on the inclusion of one summability method in the other for scalar sequences. This theorem is proved in \S\ref{S:Scalar-inclusionSommBanachSpace}, and Corollary~\ref{C:Mainresult2AbelDivergence} 
is deduced in \S\ref{SS:DivergenceOfPowerSeriesMeans}.

\section{Sequence-to-function summability methods}\label{S:SummMEthodBanachSpaces}
We start by defining some terminology in summability theory. Our main references are Hardy \cite{hardy1949} and Boos \cite{Boos2000}. Throughout this section, we let $X$ denote a Banach space over the complex numbers $\bC$ and  $\Vert \cdot \Vert_X$ be its norm. We denote by $c (X)$  the space of convergent sequences in $X$. 

A \textit{sequence-to-function summability method} $K$ is given by a sequence $(k_n)_{n \geq 0}$ of functions $k_n : [0, R) \ra \bC$ where $R \in (0, \infty ]$. The \textit{$K$-means} of a sequence of vectors $x := (x_n)_{n \geq 0} \subset X$ are defined by the following formal series
		\begin{align*}
		K_r [x] := \sum_{n \geq 0} k_n (r) x_n  \quad  (r \in [0, R ) ).
		\end{align*}
	We say that a sequence $x := (x_n)_{n \geq 0} \subset X$ is \textit{$K$-summable} or \textit{summable by the method $K$} if the series defining $K_r [x]$ converges for every $r \in [0, R)$ and, moreover, $K_r[x]$ converges in norm to some $y \in X$ as $r \ra R^-$. We say that a sequence-to-function summability method $K$ is \textit{regular} if, 
	whenever $X$ is a Banach space and 
	 $(x_n)_{n \geq 0} \in c (X)$, then $(x_n)_{n \geq 0}$ is $K$-summable and $\lim_{r \ra R^-} K_r[x] = \lim_{n \ra \infty} x_n$. Necessary and sufficient conditions for a sequence-to-function summability method to be regular are given in the following theorem, which is a slight modification of Theorem 5 in \cite[p.~49]{hardy1949}.
		\begin{theorem}\label{T:conditionForRegularity}
		Let $X$ be a Banach space. A sequence-to-function summability method $K$ is regular if and only if the following conditions hold :
			\begin{itemize}
			\item there exists an $R_0 \in [0, R)$ such that the function $r \mapsto \sum_{n \geq 0} |k_n (r)|$ is uniformly bounded on $[R_0, R)$;
			\item for each $n \geq 0$, we have $k_n (r) \ra 0$ as $r \ra R^-$;
			\item the function $k(r) := \sum_{n \geq 0} k_n (r)$ converges to $1$ as $r \ra R^-$.
			\end{itemize}
		\end{theorem}
		\begin{proof}
		By considering the homeomorphism $r \mapsto \log \op \frac{R}{R - r} \fp$ from $[0, R )$ onto $[0, \infty )$, we may restrict our attention to the situation to the interval $[0, \infty )$.
		
	Suppose that the method $K$ is regular. Since $\bC$ embeds isometrically into $X$ and the functions $k_n (r)$ are complex-valued, $K$ is regular for the space of convergent complex-valued sequences. Then the result follows from the classical case $X = \bC$ (see Theorem 5 in \cite[p.~49]{hardy1949}).
		 
		If the conditions are satisfied, then elementary estimates using the conditions show that $K_r[x] \ra \lim_{n \ra \infty} x_n$ as $r \ra R^-$ for any sequence $(x_n)_{n \geq 0} \in c (X)$.
		\end{proof}
		
		Let $K$ and $H$ be two sequence-to-function summability methods.
		We say that $K$ is \textit{included} in $H$, denoted by $K \subseteq H$, if, whenever $(x_n)_{n \geq 0}$ is a $K$-summable sequence in a Banach space $X$, then $(x_n)_{n \geq 0}$ is also $H$-summable and 
		$$
		\lim_{r \ra R^-} H_r [x] = \lim_{r \ra R^-} K_r[x] .
		$$
	If $K \subseteq H$ and $H \subseteq K$, we say that the two  summability methods are \textit{equivalent}. If $X = \bC$, all of the aforementioned definitions will be preceded by the word ``scalar''. For example, when we write ``the summability method $K$ is scalar-equivalent to the summability method $H$'', this means that they are equivalent when applied to scalar-valued sequences, that is $(x_n)_{n \geq 0} \subset \bC$.
	
	We now give some examples of sequence-to-function summability methods.
		 
	\subsection{Matrix summability methods}\label{Ex:ExempleMatrixSummabilityMethods}
	Let $R = \infty$ and, for each $n \geq 0$, let the function $k_n$ be constant on each interval $[m, m + 1 )$, where $m \geq 0$ is an integer. Then, for each $r \in [m, m + 1 )$, the expression of the $K$-mean becomes
		\begin{align*}
		K_r [x] = \sum_{n \geq 0} k_n (r) x_n = \sum_{n \geq 0} k_n (m) x_n .
		\end{align*}
	Hence the method given by the sequence of functions $k_n$ can be viewed as an infinite matrix $(k_n (m))_{m, n \geq 0}$. In these circumstances, we call the summability method a \textit{matrix summability method}. The necessary and sufficient conditions for a matrix summability method to be regular are attributed to Silverman and Toeplitz (see \cite[p.~43]{hardy1949}). The conditions in Theorem \ref{T:conditionForRegularity} now become: 
		\begin{itemize}
		\item there is a number $M > 0$ such that $\sum_{n \geq 0} |k_n (m)| \leq M$ for every $m \geq 0$;
		\item for each $n \geq 0$, $\lim_{m \ra \infty} k_n (m) = 0$;
		\item we have $\sum_{n \geq 0} k_n (m) \ra 1$, as $m \ra \infty$.
		\end{itemize}
	There is a generalization of the Silverman-Toeplitz Theorem, due to Robinson, to matrix summability methods that are given by a matrix $(C_{m, n})_{m, n \geq 0}$, where $C_{m, n}$ are bounded linear operators on $X$ (see \cite[Theorem IV]{Robinson1950}).
	
	\subsection{Power-series summability methods}
	Let $p(r) := \sum_{n \geq 0} p_n r^n $ be a power series with a radius of convergence $R_p > 0$, where $p_0 > 0$ and $p_n \geq 0$ for $n \geq 1$. Following \S3.6 of \cite{Boos2000}, we say that a sequence $x := (x_n)_{n \geq 0} \subset X$ is \emph{summable by the power-series method} $(p)$, or is \emph{$P$-summable}, if the series
		\begin{align*}
		P_r [x] := \frac{1}{p (r)} \sum_{n \geq 0} p_n x_n r^n
		\end{align*}
	converges for each $r \in [0, R)$ and there exists a $y \in X$ such that
		\begin{align*}
		P_r[x] \ra y \quad (r \ra R^- ).
		\end{align*}
	According to Theorem \ref{T:conditionForRegularity}, a power-series method $(p)$ is regular if and only if $p(r) \ra \infty$ as 
	$r \ra R_p^-$. 
	The \textit{Abel summability method} is a special case of the power-series method $(p)$ with 
	$p(r) = (1 - r)^{-1}$ 
	and $r \in [0, 1)$. The $A_r$-means are
		\begin{align*}
		A_r [x] = (1 - r) \sum_{n \geq 0} x_n r^n  \quad (0 \leq r < 1 ) .
		\end{align*}

	In this paper, we apply a power-series method to the sequence of partial sums $(s_n [f])_{n \geq 0}$ of the Taylor expansion of a function $f \in \hol ( \bD )$. The expression of the means defined by a power-series method $(p)$ are
		\begin{align*}
		P_r [f] (z) := \frac{1}{p (r)} \sum_{n \geq 0} p_n s_n[f] (z ) r^n ,
		\end{align*}
	where $p$ has a radius of convergence $R_p \geq 1$. Since $|s_n [f] (z)| \leq C (R) R^n$ for any $R > 1$ and some constant $C(R) > 0$, the function $P_r[f]$ is well-defined for each $z \in \bD$. Also, the series defining $P_r [f]$ converges uniformly on compact subsets of $\bD$, and thus it defines a function holomorphic on all of $\bD$.
	
	A useful power-series method is the logarithmic method. As we shall see later in the paper, this power-series method is convenient because it contains many other summability methods. The logarithmic method is the power-series method associated with the power series
		\begin{align*}
		l (r) := \sum_{n \geq 0} \frac{r^n}{n + 1} = \frac{1}{r} \log \frac{1}{1 - r} \quad (0 \leq r < 1 ) .
		\end{align*}
	Thus, the expression of the logarithmic mean of the partial sums $s_n[f]$ is
		\begin{align*}
		L_r[f] (z) := \frac{r}{\log \frac{1}{1 - r}} \sum_{n \geq 0} \frac{s_n[f] (z)}{n + 1} r^n \quad (0 \leq r < 1 ).
		\end{align*}
	
	We will also study the generalized Abel methods. Given a number $\alpha > -1$, the generalized Abel means of order $\alpha$ are associated with the power series
		\begin{align*}
		a_r^{\alpha} (r) := \sum_{n \geq 0} \binom{n + \alpha}{\alpha} r^n = \frac{1}{(1 - r)^{1 + \alpha}} \quad (r \in [0, 1)) .
		\end{align*}
	Applied to the partial sums $(s_n[f])_{n \geq 0}$ of the Taylor series of $f \in \hol (\bD )$, the expression of the mean $A_r^\alpha [f]$ is
		\begin{align*}
		A_r^{\alpha}[f] (z) := (1 - r)^{\alpha + 1} \sum_{n \geq 0} \binom{n + \alpha}{\alpha} s_n[f] (z) r^n . 
		\end{align*}
	If $\alpha < \beta$, then the generalized Abel method of order $\beta$ is scalar-included in the generalized Abel method of order $\alpha$ (see \cite[Theorem 2]{Borwein1957a}). When $\alpha = 1$, we obtain the classical Abel means. As we mention earlier in the paper, the expression of $A_r^1[f]$ can be rearranged to give
		\begin{align*}
		A_r^1 [f] (z) = \sum_{n \geq 0} a_n r^n z^n ,
		\end{align*}
	which is the dilate $f_r$ of $f \in \hol (\bD )$. We present a pointwise relation between the logarithmic means and the dilates of a function. It will be used later in \S\ref{SS:DivergenceOfLogInHb} and its proof is straightforward.
		\begin{lemma}
		For any $f \in \hol (\bD )$, $r \in [0, 1)$ and $z \in \bD$, we have
			\begin{equation}
			L_r[f] (z) = \frac{1}{\log \frac{1}{1 - r}} \int_0^r \frac{f_t (z)}{1 - t} \, dt . \label{E:IntegralFormulaLogAndAbel}
			\end{equation}
		\end{lemma}

\section{Background on de Branges--Rovnyak spaces}
	Let $\phi \in L^{\infty} (\bT )$. The Toeplitz operator $T_{\phi} : H^2 \ra H^2$ with symbol $\phi$ is defined as
		\begin{align*}
		T_{\phi} f := P_+ (\phi f )  \quad (f \in H^{2}),
		\end{align*}
	where $P_+ : L^2 (\bT ) \ra H^2$ is the orthogonal projection of $L^2 (\bT )$ onto $H^2$. This is clearly a bounded operator with $\Vert T_{\phi} \Vert \leq \Vert \phi \Vert_{L^{\infty} (\bT )}$. (In fact, $\Vert T_{\phi} \Vert = \Vert \phi \Vert_\infty$ by a theorem of Brown and Halmos, but we do not need this here.) The adjoint of $T_\phi$ is $T_{\overline{\phi}}$. If $\phi \in H^{\infty}$, then $T_\phi$ is simply the operator of multiplication by $\phi$. We now introduce the de Branges--Rovnyak space associated to a function $b \in H^{\infty}$, where $\Vert b \Vert_{\infty} \leq 1$. This is Sarason's definition taken from \cite{SarasonHSUD1994}.
		\begin{definition}\label{D:DefinitionOfHb}
		Let $b \in H^{\infty}$ with $\Vert b \Vert_{\infty} \leq 1$. The associated \textit{de Branges--Rovnyak space}, denoted by $\mathcal{H} (b)$, is the range space $(I - T_bT_{\overline{b}})^{1/2} H^2$ equipped with the following inner product
			\begin{align*}
			\left\langle (I - T_bT_{\overline{b}})^{1/2} f , (I - T_bT_{\overline{b}} )^{1/2} g \right\rangle_b := \left\langle f , g \right\rangle_2,
			\end{align*}
		where $f , g \in H^2 \ominus \ker (I - T_bT_{\overline{b}})^{1/2}$.
		\end{definition}
	The definition of the inner product $\left\langle \cdot , \cdot \right\rangle_b$ makes the operator $(I - T_bT_{\overline{b}})^{1/2} : H^2 \ra H^2$ a partial isometry from $H^2$ onto $\mathcal{H} (b)$. Its cousin, the space $\mathcal{H} (\overline{b})$, is defined similarly by interchanging the roles of $b$ with $\overline{b}$ in the above definition.
	
	There is a close relationship between $\mathcal{H} (b)$ and $\mathcal{H}(\overline{b})$. This is the content of the next theorem.
		\begin{theorem}\cite[\S II-2]{SarasonHSUD1994}.
		A function $f \in H^2$ belongs to $\mathcal{H} (b)$ if and only if $T_{\overline{b}} f$ belongs to $\mathcal{H} (\overline{b})$. In this case, we have that
			\begin{align*}
			\Vert f \Vert_b^2 = \Vert f \Vert_2^2 + \Vert T_{\overline{b}} f \Vert_{\overline{b}}^2 \text{.}
			\end{align*}
		\end{theorem}
	
	The structure of a de Branges--Rovnyak space depends strongly on whether $b$ is an extreme or a non-extreme point of the unit ball of $H^\infty$. 
		\begin{theorem}\cite[Theorem 2.2]{El-Fallah2016}.\label{T:densitypolynomialsAndNOnextreme}
		Let $b \in H^{\infty}$ with $\Vert b \Vert_\infty \leq 1$. The following statements are equivalent:
			\begin{enumerate}[i)]
			\item $b$ is a non-extreme point of the unit ball of $H^{\infty}$;
			\item $\mathcal{H} (b)$ contains all functions holomorphic in a neighbourhood of $\overline{\bD}$;
			\item $\mathcal{H} (b)$ contains all polynomials;
			\item polynomials are dense in $\cH (b)$.
			\end{enumerate}
		\end{theorem}
	From now on, we shall simply say that $b$ is ``extreme'' or ``non-extreme'' to indicate that $b$ is correspondingly an extreme or non-extreme point of the unit ball of $H^{\infty}$.
	The function $b$ is non-extreme if and only if $\log (1 - |b|^2 ) \in L^1 (\bT )$ (see \cite[Theorem 7.9]{DurenTHpS1970}). In this case, there exists a unique outer function $a \in H^{\infty}$, normalized so that $a(0) > 0$, such that $|a|^2 + |b|^2 = 1$ a.e. on $\bT$. We call $(b, a )$ the  \textit{Pythagorean pair} associated to $b$. 
		
	There is a useful characterization of $\mathcal{H} (b)$ when $b$ is non-extreme. 
		\begin{theorem}\cite[\S IV-1]{SarasonHSUD1994}\label{T:CharacterizationNonextremeHb}.
		Let $b$ be non-extreme, let $(b,a)$ be the corresponding Pythagorean pair, and let $f \in H^2$. Then $f \in \mathcal{H} (b)$ if and only if $T_{\overline{b}}f \in T_{\overline{a}} H^2$. In this case, there exists a unique function $f^+ \in H^2$ such that $T_{\overline{b}}f = T_{\overline{a}} f^+$, and
			\begin{align*}
			\Vert f \Vert_b^2 = \Vert f \Vert_2^2 + \Vert f^+ \Vert_2^2 \text{.}
			\end{align*}
		\end{theorem}
	
	The authors of \cite{GuillotChevrotRans2010} obtained an explicit formula for the $\mathcal{H} (b)$-norm of functions $f$ holomorphic in a neighbourhood of $\overline{\bD}$. Notice that, if $(b, a)$ is a pair, then $\phi := b/a \in N^+$, the Smirnov class.
		\begin{theorem}\cite[Theorem 4.1]{GuillotChevrotRans2010}
		Let $(b, a)$ be a pair, and let $\phi := b/a$, say $\phi (z) = \sum_{j \geq 0} c_n z^n$. Let $f$ be holomorphic in a neighbourhood of $\overline{\bD}$ with expansion $f(z) = \sum_{j \geq 0} a_n z^n$. Then the series $\sum_{j \geq 0} a_{j + n} \overline{c}_j$ converges absolutely for each $n$ and 
			\begin{align}
			\Vert f \Vert_b^2 = \sum_{n \geq 0} |a_n|^2 + \sum_{n \geq 0} \Big\lvert \sum_{j \geq 0} a_{j + n} \overline{c}_j\Big\rvert^2 . \label{E:FormulaForHbNormHolNeighborhood}
			\end{align}
		\end{theorem}
		
	We remark two consequences of this result that will be useful in what follows.
	
	\begin{corollary}\label{C:R^n}
		Let $b$ be non-extreme and let $f\in\cH(b)$.
		Then, for each $R>1$, the partial sums of the Taylor series of $f$ satisfy $\|s_n[f]\|_b=O(R^n)$ as $n\to\infty$.
		\end{corollary}
		
		\begin{proof}
		Fix $S$ with $1<S^2<R$.
		Let $(b,a)$ be the Pythagorean pair corresponding to $b$, and let $\phi:=b/a$. 
		Since both $f$ and $\phi$ are holomorphic on~$\bD$, 
		their respective Taylor coefficients satisfy $a_j=O(S^j)$ and $c_j=O(S^j)$ as $j\to\infty$. 
		Feeding this information into the formula 
		\eqref{E:FormulaForHbNormHolNeighborhood}, applied to $s_N[f]$ in place of $f$, we obtain
		\[
		\|s_N[f]\|_b^2=O((N+1)S^{2N})+O((N+1)^3S^{4N}) \quad(N\to\infty).
		\]
		This implies that $\|s_N[f]\|_b=O(R^N)$.
	\end{proof}
	
		\begin{corollary}\label{C:NormEstimateDilatesInHb}
		Let $(b, a)$ be a pair, and let $\phi := b/a$, say $\phi (z) = \sum_{j \geq 0} c_n z^n$. 
		Let $f(z) = \sum_{n \geq 0} a_n z^n$ be a function in $H^2$ and let $r \in [0, 1)$. 
		Then there is a constant $C(\phi, r)$, which depends only on $\phi$ and $r$, such that
			\begin{align*}
			\Vert f_r \Vert_{b}^2 \leq C (\phi , r) \Vert f \Vert_{H^2}^2 .
			\end{align*}
		The dependence $r \mapsto C(\phi, r)$ can be chosen to be increasing.
		\end{corollary}
		\begin{proof}
		Since $f_r (z) := \sum_{n \geq 0} a_n r^n z^n$, using formula \eqref{E:FormulaForHbNormHolNeighborhood}, we get
			\begin{align*}
			\Vert f_r \Vert_b^2 &= \sum_{n \geq 0} r^2 |a_n|^2 + \sum_{n \geq 0} \Big\lvert \sum_{j \geq 0} r^{j + n} a_{j + n} \overline{c}_j \Big\rvert^2 \\
			& \leq \Vert f \Vert_{H^2}^2 + \sum_{n \geq 0} \Big( \sum_{j \geq 0} r^{j + n} |a_{j + n}| |c_j| \Big)^2.
			\end{align*}
		By the Cauchy-Schwarz inequality, we obtain
			\begin{align*}
			\Vert f_r \Vert_b^2 & \leq \Vert f \Vert_{H^2}^2 + \sum_{n \geq 0} \Big( \sum_{j \geq 0} r^{j + n} |a_{j + n}|^2 \Big) \Big( \sum_{j \geq 0} r^{j + n} |c_j|^2 \Big) \\
			&= \Vert f \Vert_{H^2}^2 + \sum_{n \geq 0} r^n \Vert (S^*)^n (f_{\sqrt{r}}) \Vert_{H^2}^2 \Big( \sum_{j \geq 0} r^j |c_j|^2 \Big) ,
			\end{align*}
		where $S^*$ is the backward shift operator on $H^2$ and $(S^*)^n$ is the $n$-fold composition of $S^*$. Since $\Vert S^* \Vert \leq 1$, we get
			\begin{align*}
			\Vert f_r \Vert_b^2 \leq \Vert f \Vert_{H^2}^2 + \frac{\sum_{n \geq 0} r^j |c_j|^2}{1 - r} \Vert f_{\sqrt{r}} \Vert_{H^2}^2 ,
			\end{align*}
		where $\sum_{j \geq 0} r^j |c_j|^2 < \infty$ since $\phi \in \hol (\bD )$. Finally, since $\Vert f_{\sqrt{r}} \Vert_{H^2}^2 \leq \Vert f \Vert_{H^2}^2$, we obtain
			$$
			\Vert f_r \Vert_b^2 \leq C (\phi , r ) \Vert f \Vert_{H^2}^2,
			$$
		as desired, where $C(\phi , r ) := 1 + \sum_{n \geq 0} r^j |c_j|^2 / (1 - r)$. It is easy to see from the definition of the constant $C(\phi , r)$ that the function $r \mapsto C(\phi , r)$ is increasing on $[0, 1)$.
		\end{proof}

	We end this section by stating an estimate obtained in the core of the proofs of Theorem 3.1 and Theorem 3.6 in \cite{El-Fallah2016}. 
	This example plays a central role in \cite{El-Fallah2016}, and will be equally important for us here.
	The exact choices of $b$ and $f \in \mathcal{H} (b)$ do not matter in our situation, so we will just state the estimate as follows.
	
		\begin{theorem}\cite[Theorems 3.1 and 3.6]{El-Fallah2016}.\label{C:EstimesDilatesHb}
		There exist a non-extreme $b$ and a function $f \in \mathcal{H} (b)$ such that $(f_r)^+(0)$ is non-negative for all $r \in [0, 1)$ and, moreover, $(f_r)^+ (0) \ra +\infty$ as $r \ra 1^-$.
		\end{theorem}

\section{Divergence of the logarithmic means}\label{SS:DivergenceOfLogInHb}
	To prove that the logarithmic means diverge in the $\cH (b)$-norm, we use another expression of the logarithmic means in terms of the classical Abel means or, equivalently, in terms of the dilates $f_r$ of a function $f \in \hol (\bD )$.
	
	The first step toward this formula is the following lemma.
	
	\begin{lemma}\label{L:ContinuityOfdilatesHb}
	Let $b$ be non-extreme, let $f \in \mathcal{H} (b)$ and let $r \in (0, 1)$. The application $F : [0, r] \ra \mathcal{H} (b)$, defined by $F(t) := f_t$, is continuous from $[0, r]$ into $\mathcal{H} (b)$.
	\end{lemma}
	\begin{proof}
	By Corollary~\ref{C:R^n}, 
	for any $R > 1$ there exists a finite positive constant $C > 0$, which depends on $R$ and $f$, such that
		\begin{align*}
		\Vert s_n [f] \Vert_b \leq C R^n .
		\end{align*}
	By choosing $R$ so that $Rr < 1$, we can ensure that the series 
	$$
	(1 - t) \sum_{n \geq 0} s_n[f] t^n \quad (0 \leq t \leq r)
	$$
	converges in $\cH (b)$-norm to some function $g^{(t)} \in \cH (b)$ on $[0, r]$. The dilates $f_t$ can be expressed pointwise as
		\begin{align*}
		f_t (z) = (1 - t) \sum_{n \geq 0} s_n[f] (z) t^n \quad (z \in \bD ),
		\end{align*}
	and therefore we must have that $g^{(t)} = f_t$, since convergence in $\cH (b)$ implies pointwise convergence. 
	
	The continuity of $F$ now follows from the scalar continuity of each map $t \mapsto t^n (1 - t)$ on $[0, r]$, $n \geq 0$.
	\end{proof}	
	
	For the next lemma, we show that the integral formula \eqref{E:IntegralFormulaLogAndAbel} linking the logarithmic means to the Abel means is also valid in $\cH (b)$. We use the Bochner integral as the definition of the vector-valued integral. For background on this topic, we refer the reader to \cite[Chapter 3, \S 1]{HillePhillips1957}.
	\begin{lemma}\label{L:IntegralFormLogMeansHb}
	Let $b$ be non-extreme and let $f \in \cH (b)$, say $f(z) = \sum_{n \geq 0} a_n z^n$. Then, for each $r \in [0, 1)$, we have $L_r[f] \in \cH (b)$ and
		\begin{align*}
		L_r[f] = \frac{1}{\log \frac{1}{1 - r}} \int_0^r \frac{f_t}{1 - t} \, dt .
		\end{align*}
	\end{lemma}
	\begin{proof}
	Fix $r \in [0, 1)$. From Lemma \ref{L:ContinuityOfdilatesHb}, the function $t \mapsto \frac{f_t}{1 - t}$ is continuous from $[0, r]$ into $\cH (b)$, and so its Bochner integral is well-defined.

	By Corollary~\ref{C:R^n} again, 
	for any $R > 1$, we have that $\Vert s_n [f] \Vert_{b} \leq C R^n$, where $C$ is a positive constant depending only on $R$ and $f$. Let $R > 1$ be chosen so that $R r < 1$. 
	
	Firstly, we compute an upper bound for the series defining $L_r[f]$: 
		\begin{align*}
		\frac{r}{\log \frac{1}{1-r}} \sum_{n \geq 0} \frac{\Vert s_n [f] \Vert_b}{n + 1} r^{n} \leq C \frac{r }{\log \frac{1}{1 - r}} \sum_{n \geq 0} \frac{(rR)^n}{n + 1} = C \frac{\log \frac{1}{1 - rR}}{R \log \frac{1}{1 - r}} .
		\end{align*}
	Therefore the series $\frac{r}{\log \frac{1}{1 - r}} \sum_{n \geq 0} \frac{s_n [f]}{n + 1} r^n$ converges absolutely in $\cH (b)$ and it defines a function in $\cH (b)$, say $g^{(r)} \in \cH (b)$. Since convergence in $\cH (b)$ implies pointwise convergence, we must also have
		\begin{align*}
		g^{(r)} (z) = \frac{r}{\log \frac{1}{1 - r}} \sum_{n \geq 0} \frac{s_n [f](z)}{n + 1} r^n \quad (z \in \bD ),
		\end{align*}
	which gives $g^{(r)} = L_r[f] \in \cH (b)$. 
	
	Secondly, we have
		\begin{align*}
		\frac{r}{\log \frac{1}{1-r}} \sum_{n \geq 0} \frac{s_n[f]}{n + 1} r^{n} = \frac{1}{\log \frac{1}{1 - r}} \sum_{n \geq 0} \int_0^r s_n [f] t^n \, dt .
		\end{align*}
	The series $\sum_{n \geq 0} s_n [f] t^n$ is absolutely and uniformly convergent in $\cH (b)$ on $[0, r]$. Therefore, the order of summation and integration can be interchanged and we get
		\begin{align*}
		L_r[f] = \frac{1}{\log \frac{1}{1 - r}} \int_0^r \sum_{n \geq 0} s_n [f] t^n \, dt = \frac{1}{\log \frac{1}{1 - r}} \int_0^r \frac{f_t}{1 - t} \, dt ,
		\end{align*}
	where the last equality comes from the fact that $f_t = (1 - t) \sum_{n \geq 0} s_n [f] t^n$.
	\end{proof}
	Using this last lemma and the estimate in Corollary \ref{C:NormEstimateDilatesInHb}, we obtain the following result. We consider the logarithmic means $L_r$ as linear operators on $\cH (b)$ defined by $L_r(f) := L_r[f]$ for each $f \in \cH (b)$.
	
		\begin{corollary}\label{C:LogmeansBoundedOpHb}
		Let $b $ be non-extreme. Then, for every $r \in [0, 1)$, the linear map $L_r : \mathcal{H} (b) \ra \mathcal{H} (b)$ is a bounded linear operator with $\|L_r\|\le \sqrt{C(\phi,r)}$,
		where $C (\phi , r)$ is the same constant as in Corollary \ref{C:NormEstimateDilatesInHb}. In fact we even have
		\begin{align}\label{E:H2est}
			\Vert L_r (f) \Vert_b \leq \sqrt{C(\phi , r)} \Vert f \Vert_{H^2}  \quad (f \in \mathcal{H}(b)).
			\end{align}
		\end{corollary}
		
		\begin{proof}
		Let $r \in [0, 1)$ and $f \in \mathcal{H} (b)$. Then, from Corollary \ref{C:NormEstimateDilatesInHb}, we have
			\begin{align*}
			\frac{1}{\log \op \frac{1}{1 - r} \fp} \int_0^r \frac{\Vert f_t \Vert_b}{1 - t} \, dt \leq \frac{\Vert f \Vert_{H^2}}{\log \op \frac{1}{1 - r} \fp} \int_0^r \frac{ \sqrt{C (\phi , t)}}{1 - t} \, dt,
			\end{align*}
		and, since $t \mapsto C(\phi, t)$ is increasing on $[0, r)$, the above expression is
			\begin{align*}
			\leq \frac{C (\phi , r ) \Vert f \Vert_{H^2}}{\log \op \frac{1}{1 - r}\fp} \int_{0}^r \frac{1}{1 - t} \, dt \leq \sqrt{C(\phi , r )} \Vert f \Vert_{H^2}.
			\end{align*}
			In combination with Lemma~\ref{L:IntegralFormLogMeansHb}, this  establishes \eqref{E:H2est}. Since  $\|f\|_{H^2}\le\|f\|_b$,
			it follows that $L_r$ is a bounded operator on $\cH(b)$ with $\|L_r\|\le	\sqrt{C(\phi,r)}$.
			\end{proof}	
	
	Now we can exploit the integral formula of $L_r[f]$ to express $(L_r[f])^+$ in terms of a certain integral involving $(f_r)^+$.
		\begin{theorem}\label{T:FormulaForL_r[f]+}
		For any $f \in \mathcal{H} (b)$ and $r \in (0, 1)$, we have that
			\begin{align}
			(L_r[f])^+ = \frac{1}{\log \frac{1}{1-r}} \int_0^r \frac{(f_t)^+}{1 - t} \, dt . \label{E:relationLrPlus}
			\end{align}
		\end{theorem}
		\begin{proof}
		Define $F(t) := \frac{f_t}{1 - t}$. This is a continuous function from $[0, r]$ into $\mathcal{H} (b)$, by Lemma \ref{L:ContinuityOfdilatesHb}. The fact that $F(t)$ is continuous from $[0, r]$ into $\mathcal{H} (b)$, combined with the formula for the norm in $\cH (b)$ in terms of $f$ and $f^+$, imply that the mapping $ t \mapsto \frac{(f_t)^+}{1 - t}$ is continuous from $[0, r]$ into $H^2$. Therefore, it is Bochner-integrable on $[0, r]$. Thus, since $T_{\overline{a}} : H^2 \ra H^2$ and $T_{\overline{b}} : H^2 \ra H^2$ are bounded operators, we get
			\begin{align*}
			T_{\overline{a}} \op \op \log \frac{1}{1 - r} \fp^{-1} \int_0^r \frac{(f_t)^+}{1 - t} \, dt \fp &= \op \log \frac{1}{1 - r} \fp^{-1} \int_0^r \frac{T_{\overline{a}}(f_t)^+}{1 - t} \, dt \\
			&= \op \log \frac{1}{1 - r} \fp^{-1} \int_0^r \frac{T_{\overline{b}} f_t}{1 - t} \, dt \\
			&= T_{\overline{b}} \op \op \log \frac{1}{1 - r} \fp^{-1} \int_0^r \frac{f_t}{1 - t} \, dt \fp .
			\end{align*}
		Thus, by the uniqueness of $(L_r[f])^+$, we get formula \eqref{E:relationLrPlus}.
		\end{proof}
	
	Now, we can prove our main result.
	\begin{proof}[Proof of Theorem \ref{T:MainresultLogDivergence}]
	Choose $b$ and $f$ as in Theorem \ref{C:EstimesDilatesHb}. Let $A > 0$ and choose $r_0 \in (0, 1)$ so that 
		\begin{align*}
		(f_t)^+ (0) \geq A \quad (r_0 < t < 1) .
		\end{align*}
	Convergence in $H^2$ implies pointwise convergence on $\bD$. Therefore, by Theorem \ref{T:FormulaForL_r[f]+}, the following equality holds:
		\begin{align*}
		(L_r[f])^+ (0) = \frac{1}{\log \frac{1}{1 - r}} \int_0^r \frac{(f_t)^+ (0)}{1 - t} \, dt .
		\end{align*} 
	Fix $r \in [0, 1 )$. Splitting the integral in two parts, from $0$ to $r_0$ and from $r_0$ to $r$, we have
		\begin{align*}
		(L_r[f])^+ (0) \geq \frac{\log \frac{1 - r_0}{1 - r}}{\log \frac{1}{1 - r}} A + \frac{1}{\log \frac{1}{ 1- r}} \int_0^{r_0} \frac{(f_t)^+ (0)}{1 - t} \, dt .
		\end{align*}
	Taking the $\liminf$ as $r \ra 1^-$, we get
		\begin{align*}
		\liminf_{r \ra 1^-} (L_r[f])^+ (0) \geq A .
		\end{align*}
	Since $A$ was arbitrary, $\liminf_{r \ra 1^-} (L_r[f])^+ (0) = \infty$. By the expression of the norm of $L_r[f]$ in $\cH (b)$, it follows that
		\begin{align*}
		\Vert L_r[f] \Vert_b \geq |(L_r[f])^+ (0)|
		\end{align*}	
	which implies that $\liminf_{r \ra 1^-} \Vert L_r[f] \Vert_b = \infty$. This concludes the proof of the theorem.
	\end{proof}
	
	\section{Scalar-inclusion and summability in Banach spaces}\label{S:Scalar-inclusionSommBanachSpace}
Before analyzing the consequences of Theorem \ref{T:MainresultLogDivergence} on other power-series methods, we take a little detour to prove an abstract theorem on scalar-inclusion of two sequence-to-function summability methods. This theorem will be the main ingredient to prove Corollary \ref{C:Mainresult2AbelDivergence}.
	
	\begin{theorem}\label{T:scalarInclusionThenOperatorBounded}
	Let $K$ and $H$ be two regular sequence-to-function summability methods.
	Let $X$ and $Y$ be Banach spaces, and let $S : X \ra Y$ and $S_n : X \ra Y$ $(n \geq 0)$ be bounded linear operators.
	Suppose that:
	\begin{itemize}
	\item $S_n(x)\to S(x)$ for all $x\in W$, where $W$ is a dense subset of $X$;
	\item $(S_n(x))_{n\ge0}$ is $K$-summable to $S(x)$ for all $x\in X$;
	\item $K$ is scalar-included in $H$.
	\end{itemize}
	Then $(S_n(x))_{n\ge0}$ is $H$-summable to $S(x)$ for all $x\in X$.
	\end{theorem}

	\begin{proof}
	Let $(k_n)_{n\ge0},(h_n)_{n\ge0}:[0,R)\to\bC$  be the functions defining the summability methods $K$ and $H$ respectively. We need to prove that, for each $x\in X$:
	\begin{enumerate}[(i)]
	\item $\sum_{n\ge0}h_n(r)S_n(x)$ converges in $Y$ for all  $r\in[0,R)$;
	\item $\|\sum_{n\ge0}h_n(r)S_n(x)-S(x)\|_Y\to0$ as $r\to R^-$.
	\end{enumerate}

	We begin with (i). Fix $r\in[0,R)$.	Given $x\in X$, the sequence $(S_n(x))_{n\ge0}$ is $K$-summable to $S(x)$. By linearity and continuity, for each $\phi\in Y^*$, the sequence $(\phi(S_n(x)))$ is $K$-summable to $\phi(S(x))$. As $K$ is scalar-included in $H$, it follows that $(\phi(S_n(x)))$ is also $H$-summable to $\phi(S(x))$. In particular, the series $\sum_{n\ge0} h_n(r)\phi(S_n(x))$ converges in $\bC$. Hence
	\[
	\sup_{m\ge0}\Bigl|\sum_{n=0}^{m}h_n(r)\phi(S_n(x))\Bigr|<\infty 
	\quad(x\in X,~\phi\in Y^*).
	\]
	In other words
	\[
	\sup_{m\ge0}\Bigl|\phi\Bigl(\sum_{n=0}^{m}h_n(r)S_n(x)\Bigr)\Bigr|<\infty
	\quad(x\in X,~\phi\in Y^*).
	\]
	Applying the Banach--Steinhaus theorem twice (once for $\phi$ and once for $x$), we obtain that
	\begin{equation}\label{E:bound}
	M:=\sup_{m\ge0}\Bigl\|\sum_{n=0}^{m}h_n(r)S_n\Bigr\|<\infty,
	\end{equation}
	where now the norm is the operator norm.

	Given $x\in X$ and $\epsilon>0$, choose $w\in W$ such that  $\|x-w\|_X<\epsilon/M$. As $S_n(w)\to S(w)$ and $H$ is a regular summability method, the sequence $(S_n(w))$ is $H$-summable to $S(w)$. In particular, the series $\sum_{n\ge0}h_n(r)S_n(w)$ converges in $Y$. It follows that  $\|\sum_{n=m_1}^{m_2}h_n(r)S_n(w)\|_Y<\epsilon$ for  all large enough $m_1,m_2$. For all such $m_1,m_2$, we then have
	\begin{align*}
	\Bigl\|\sum_{n=m_1}^{m_2}h_n(r)S_n(x)\Bigr\|_Y
	&\le \Bigl\|\sum_{n=m_1}^{m_2}h_n(r)S_n(x-w)\Bigr\|_Y+\Bigl\|\sum_{n=m_1}^{m_2}h_n(r)S_n(w)\Bigr\|_Y\\
	&\le \Bigl\|\sum_{n=m_1}^{m_2}h_n(r)S_n\Bigr\|\|x-w\|_X+\epsilon\\
	&\le 2M(\epsilon/M)+\epsilon= 3\epsilon.
	\end{align*}
	This shows that the series $\sum_{n=0}^\infty h_n(r)S_n(x)$ is Cauchy, and therefore it converges in $Y$, thereby completing the proof of (i).

Now we turn to (ii). For each $r\in[0,R)$, define an operator $S^H_r:X\to Y$ by
	\[
	S^H_r(x):=\sum_{n\ge0}h_n(r) S_n(x) \quad(x\in X).
	\]
	By (i) the series converges, so $S^H_r$ is well-defined and linear. Furthermore, it follows easily from \eqref{E:bound} that $S^H_r$ is a bounded linear operator from $X$ into $Y$. As we saw in (i), for each $x\in X$ and $\phi\in Y^*$, the sequence $\phi(S_n(x))$ is $H$-summable to $\phi(S(x))$,
	in other words
	\begin{equation}
	\phi(S^H_r(x))\to\phi(S(x)) \quad(r\to R^-). \label{E:WeakConvergence}
	\end{equation}

	We want to prove that $S_r^H (x) \to S(x)$ as $r \to R^-$. To do so, let $(r_j)_{j \geq 0}$ be a sequence in $[0, R)$ such that $r_j \to R^-$ ($j \to \infty$). We will prove that $S_{r_j}^H (x) \ra S(x)$ as $j \to \infty$. By \eqref{E:WeakConvergence}, we have $\phi (S^H_{r_j} (x) )\to \phi (S(x))$ as $j \to \infty$ for each $x \in X$ and $\phi \in Y^*$. By the Banach--Steinhaus theorem, again applied twice, it follows that
	\[
	N:=\sup_{j \geq 0}\|S^H_{r_j}\|<\infty.
	\]

	Given $x\in X$ and $\epsilon>0$, choose $w\in W$ such that 
	\[
	\|x-w\|_X<\epsilon/\max\{N,\|S\|\} .
	\]
	By regularity of $H$, we have
	\[
	S^H_{r_j}(w)\to S(w) \quad (j \to \infty).
	\]
	Hence $\|S^H_{r_j}(w)-S(w)\|_Y<\epsilon$ for all $j$ sufficiently large. For all such $j$, we then have
	\begin{align*}
	\|S^H_{r_j}(x)-S(x)\|_Y
	&\le \|S^H_{r_j}(x-w)\|_Y+\|S^H_{r_j}(w)-S(w)\|_Y+\|S(w-x)\|_Y\\
	&\le N(\epsilon/N)+\epsilon+\|S\|(\epsilon/\|S\|)=3\epsilon.
	\end{align*}
	We conclude that $S^H_{r_j}(x)\to S(x)$ as $j \ra \infty$, completing the proof of (ii).
	\end{proof}
	
\section{Consequences for other power-series methods}\label{SS:DivergenceOfPowerSeriesMeans}
Our final goal is to prove Corollary \ref{C:Mainresult2AbelDivergence}. Recall from \S\ref{S:SummMEthodBanachSpaces} that the generalized Abel means of order $\alpha > -1$ applied to a function $f \in \hol (\bD )$ are
		\begin{align*}
		A^{\alpha}_r[f](z) := (1 - r)^{1 + \alpha} \sum_{n \geq 0} \binom{n + \alpha}{\alpha} s_n (z) r^n \quad (z \in \bD , \, r \in [0, 1) )\text{.}
		\end{align*}
	Let $f \in \cH (b)$ and let $r \in [0, 1)$. 
	By Corollary~\ref{C:R^n},
	 for any $R > 1$, we have $\Vert s_n[f] \Vert_b \leq C R^n$ for some constant $C$ depending only on $R$ and $f$. From this, it is easy to see that the series defining $A^{\alpha}_r[f]$ converges absolutely to some function $g^{(r)} \in \cH (b)$ and, since convergence in $\cH (b)$ implies pointwise convergence, we get $A^{\alpha}_r [f] = g^{(r)} \in \cH (b)$. Thus, we may view the generalized Abel mean as a linear operator on $\mathcal{H} (b)$ defined by
		\begin{align*}
		A^{\alpha}_r (f) := A^{\alpha}_r [f]  \quad (f \in \mathcal{H} (b), \, r \in [0, 1) ).
		\end{align*}
	By the closed-graph theorem, it is a bounded linear operator on $\cH (b)$. 
	
	To use Theorem \ref{T:scalarInclusionThenOperatorBounded}, we need the following relation between the generalized Abel methods and the logarithmic method. This result can be found in \cite[\S 5]{Borwein1957b}.
	\begin{theorem}\label{T:FactsAboutAalpha}
	All generalized Abel methods of order $\alpha$, with $\alpha > -1$, are scalar-included in the logarithmic method.
	\end{theorem}
	
	We are now ready to prove Corollary \ref{C:Mainresult2AbelDivergence}.
	\begin{proof}[Proof of Corollary \ref{C:Mainresult2AbelDivergence}]
		Let $b$ be as in Theorem \ref{T:MainresultLogDivergence}. Suppose, if possible, that, for every $f \in \cH (b)$,
		$$
		A_r^\alpha (f) \ra f , \quad (r \ra 1^-).
		$$
		
	To apply Theorem \ref{T:scalarInclusionThenOperatorBounded}, let $X = Y = \cH (b)$, let $W$ be the set of polynomials, let $R = 1$, let $K$ and $H$ be the generalized Abel method of order $\alpha > -1$ and the logarithmic method respectively, let $S_n := s_n$, the operator that maps each function to the $n$-th partial sum of its Taylor expansion, and let $S := I$ be the identity on $\cH (b)$. 
	
	By Theorem \ref{T:FactsAboutAalpha}, the Abel method of order $\alpha > -1$ is scalar-included in the logarithmic method. The logarithmic method is also a regular sequence-to-function summability method because the function $\frac{1}{r}\log (\tfrac{1}{1-r}) \ra \infty$ as $r \ra 1^-$. Moreover, by Theorem \ref{T:densitypolynomialsAndNOnextreme}, the set of polynomials $W$ is dense in $\cH (b)$ since $b$ is non-extreme.
	
	It remains to verify the first condition of Theorem \ref{T:scalarInclusionThenOperatorBounded}, that is, $s_n (p) \ra p$ for any $p \in W$. This is clear since, for $p \in W$, with 
	$$
	p(z) = \sum_{n = 0}^N a_n z^n \quad (z \in \bD ),
	$$
we have $s_n (p) = s_n [p] = p$ if $n \geq N$. 
		
	Therefore, by Theorem \ref{T:scalarInclusionThenOperatorBounded}, we infer that $(s_n (f))_{n \geq 0}$ is summable by the logarithmic method for every $f \in \cH (b)$. This contradicts Theorem \ref{T:MainresultLogDivergence}. Therefore, there exists a function $f \in \cH (b)$ such that $A_r^\alpha (f) \not\ra f$ as $r \ra 1^-$.
		\end{proof}
		
		In fact, Corollary \ref{C:Mainresult2AbelDivergence} generalizes to a whole family of power series methods. This generalization is a consequence of the following inclusion theorem due to Borwein.
		
		\begin{theorem}\cite[Theorem A]{Borwein1957b}\label{T:BorweinEquivalentPowermethods}
		Let $p (x) = \sum_{n \geq 0} p_n x^n$ be a power series with a radius of convergence $R_p > 0$. Let $q(x) = \sum_{n \geq 0} q_nx^n$ be another power series with a radius of convergence $R_q = R_p$. Suppose that there exist an integer $N$, a finite signed measure $\mu$ on the interval $[0, 1]$ and a number $\delta \in (0, 1]$ such that, for all $n \geq N$,
			\begin{enumerate}
			\item $p_n = q_n \int_0^1 t^n \, d\mu (t)$;
			\item $\int_0^1 t^n \, d \mu (t) \geq \delta \int_0^1 t^n \, |d\mu (t)|$.
			\end{enumerate}
		Then the power-series method $(q)$ is scalar-included in the power-series method $(p)$.
		\end{theorem}
		
		Consider the power-series method defined by the power series $p(r) := \sum_{n \geq 0} p_n r^n$ with a radius of convergence $R_p = 1$. We suppose that, for the coefficients $(p_n)_{n \geq 0}$, there exist an integer $N$, a finite signed measure $\mu$ on $[0,1]$ and a number $\delta \in (0, 1]$ such that, for every $n \geq N$,
			\begin{enumerate}
			\item[(A)] $\frac{1}{n + 1} = p_n \int_0^1 t^n \, d\mu (t)$;
			\item[(B)] $\frac{1}{n + 1} \geq \delta p_n \int_0^1 t^n \, |d \mu (t)|$.
			\end{enumerate}
		As an explicit example, the coefficients $p_n := \binom{n + \alpha}{\alpha}$ with $\alpha > -1$ satisfy these requirements (see \cite{Borwein1957b} for details).
		
		Recall that, applied to the partial sums of the Taylor expansion of $f (z) = \sum_{n \geq 0} a_n z^n$, the means defined by the power-series method $(p)$ are
			\begin{align*}
			P_r[f]:= \frac{1}{p(r)} \sum_{n \geq 0} p_n s_n[f] r^n \quad ( 0 \leq r < 1 ).
			\end{align*}
		Using formula \eqref{E:FormulaForHbNormHolNeighborhood} again, we can prove that the series defining $P_r[f]$ converges, in $\cH (b)$, to some function $g^{(r)} \in \cH (b)$ and that $P_r[f] = g^{(r)} \in \cH (b)$. Also, by the closed-graph theorem, the linear operator $P_r : \mathcal{H} (b) \ra \mathcal{H} (b)$, defined by $P_r(f) := P_r[f]$, is bounded for each $r \in [0, 1)$.
		\begin{theorem}
		Let $p (r) := \sum_{n \geq 0} p_n r^n$ be a power series with a radius of convergence $R_p = 1$. Suppose that the coefficients $(p_n)_{n \geq 0}$ satisfy the conditions (A) and (B) above, with $N \geq 0$ an integer, $\mu$ a finite signed measure on $[0, 1]$ and $\delta \in (0, 1]$. Then there exist a non-extreme point $b \in H^\infty$ and a function $f \in \cH (b)$ such that $P_r [f] \not\ra f$ in $\cH (b)$ as $r \ra 1^-$.
		\end{theorem}
		\begin{proof}
		The proof is similar to Corollary \ref{C:Mainresult2AbelDivergence}. The only difference is that we use Theorem \ref{T:BorweinEquivalentPowermethods} instead of Theorem \ref{T:FactsAboutAalpha} to show that the power-series method defined by $(p)$ is scalar-included in the logarithmic method. 
		\end{proof}
		
		There are plenty of other summability methods which are included in the logarithmic method that we did not treat here. For example, in \cite{BorweinWatson1981},  it was shown that the generalized Borel method $B^{\alpha, \beta}$, with its means defined by
	$$
	B_x^{\alpha, \beta} [f] := \sum_{n \geq N} \frac{s_n[f]}{\Gamma (\alpha n + \beta )} x^{\alpha n + \beta - 1},
	$$
where $x > 0$, $\alpha > 0$, $\beta \in \bR$ and $N$ is a non-negative integer such that $\alpha N + \beta > 1$, is scalar-included in the logarithmic method. With similar techniques, we can prove that $B^{\alpha, \beta}_x [f] \in \cH (b)$ for each $f \in \cH (b)$ and, by using Theorem \ref{T:scalarInclusionThenOperatorBounded}, we can prove that there exist a non-extreme $b$ and a function $f \in \cH (b)$ such that $B^{\alpha, \beta}_x [f] \not \ra f$ in $\cH (b)$ as $x \ra \infty$.

\section{An open problem}
We have studied some classes of sequence-to-function summability methods in this article and we proved that none of them works in general for de Branges--Rovnyak spaces. Moreover, in \cite{Mashreghi2020Failure}, the authors showed that there exist a Hilbert space of analytic functions $\cH$ continuously embedded in $\hol (\bD )$ and a function $f \in \cH$ such that $f$ lies outside the closed linear span of $\oa s_n[f] \, : \, n \geq 0 \fa$. In particular, no summability methods applied to the partial sums $s_n[f]$ would work to make the Taylor series summable to $f$ in $\cH$. This motivates the following problem.
		\begin{problem}\label{P:worstHb}
		Does there exist a non-extreme function $b$ and a function $f \in \mathcal{H} (b)$ such that $f$ lies outside the closed linear span of $\oa s_n[f] \, : \, n \geq 0 \fa$?
		\end{problem}
		
		\section*{Acknowledgement} We are grateful to the anonymous referee for reading the article carefully and for making several helpful suggestions to improve the exposition.

\bibliographystyle{plain}
\bibliography{biblio}

\end{document}